\documentclass[review,3p,times,12pt]{elsarticle}

\usepackage{amssymb}
\usepackage{latexsym}
\usepackage{amsmath}
\usepackage{amssymb}
\usepackage{pifont}
\usepackage[colorlinks]{hyperref}
\usepackage{amsthm}
\usepackage{comment}
\newtheorem{theorem}{Theorem}[section]

\newtheorem{corollary} [theorem]{Corollary}
\newtheorem{proposition}[theorem]{Proposition}

{\theoremstyle{definition}\newtheorem{example}{Example}[section]}
{\theoremstyle{definition}\newtheorem{remark}{Remark}[section]}

\newdefinition{defn}[thm]{Definition}
\newproof{pf}{Proof}

\begin{document}

\begin{frontmatter}

\title{Products of functions with bounded ${\rm Hess}^+$ complement}

\author[ab]{Andi Brojbeanu\corref{cor1}}
\ead{andi_bro@yahoo.com}
\author[ab]{Cornel Pintea\fnref{p1}} 
\ead{cpintea@math.ubbcluj.ro}

\cortext[cor1]{Corresponding author.}

\address[ab]{Babe\c{s}-Bolyai University, Faculty of Mathematics and Computer
Science, 400084 M. Kog\u alniceanu 1, Cluj-Napoca, Romania}

\begin{abstract}
\noindent We denote by  ${\rm Hess}^+(f)$ the set of all points $p\in\mathbb{R}^n$ such that the Hessian matrix $H_p(f)$ of the $C^2$-smooth function 
		$f:\mathbb{R}^n\longrightarrow\mathbb{R}$ is positive definite. In this paper we provide a class of norm-coercive polynomial functions with large ${\rm Hess}^+$ regions, as their ${\rm Hess}^+$ complements happen to be bounded. A detailed analysis concerning the  ${\rm Hess}^+$ region of a particular polynomial function along with some basic properties of its level curves, such as regularity, connectedness and convexity, is also provided.  For such functions we also prove several properties, such as connectedness and convexity, of their level sets for sufficiently large levels. Apart from the mentioned source of such examples we provide some sufficient conditions on two functions $f,g:\mathbb{R}^2\longrightarrow\mathbb{R}$ with  bounded ${\rm Hess}^+$ complements whose product $fg$ keeps having  bounded ${\rm Hess}^+$ complement as well.
\end{abstract}


\end{frontmatter}

\section{Introduction}
	We noticed in \cite{Pintea-Tofan} that the polynomial function \[f_a:\mathbb{R}^2\longrightarrow\mathbb{R}, \ f_a(x,y)=(x^2+y^2)^2-2a^2(x^2-y^2)\]
	(whose regular level sets are the  Cassini's ovals and the critical zero level set is the Bernoulli's lemniscate) is not convex, as it has, for instance, 
	exactly two global minimum points. In fact $f_a$ is a Morse function with precisely three critical points, two of whom have index zero and one of them has index one.
	However, a detailed analysis done there, shows that starting with a certain level and going higher and higher we only get convex regular levels bounding convex sublevel sets. This is due to the curvature function which preserves its sign along such regular level sets. The first such regular level set is the one which is completely contained in the region ${\rm Hess}^+(f_a)$, where the Hessian matrix of $f_a$ is positive definite, and all levels above 
	it keep being convex and completely contained in ${\rm Hess}^+(f)$. Note that the same facts are valid for the polynomial function \[g_a:\mathbb{R}^2\longrightarrow\mathbb{R}, \ g_a(x,y)=(x^2+y^2)^2+2a^2(x^2-y^2)\] as can be similarly proved.
	The terminology of level and sublevel sets at the level $y\in\mathbb{R}$
	is used for the inverse images of types $f^{-1}(y)$ and $f^{-1}((-\infty,y])$ of real-valued functions 
	$f:\mathbb{R}^n\longrightarrow\mathbb{R}$, respectively. If the function $f$ is additionally convex, then its sublevel sets are obviously convex.
	Although the level sets are usually curved hypersurfaces, impossible to be convex subsets of $\mathbb{R}^n$, they could sometimes bound 
	convex sublevel sets and they are said to be convex in such a case \cite{Rybnikov} (see also  \cite[p. 175]{Montiel-Ros}). 
	Another way to consider convexity for regular hypersurfaces of $\mathbb{R}^n$ consists in 
	their quality to stay on the same side of all of their tangent hyperplanes \cite[p. 174]{Montiel-Ros}, \cite[p. 37]{Carmo}. 
	The property of the level sets of $f_a$ to be contained in ${\rm Hess}^+(f_a)$, 
	starting with a certain level, is due to the fact that its complement $\mathbb{R}^2\setminus
	{\rm Hess}^+(f_a)$ is bounded, as the sublevel sets of $f_a$ grow over and over and cover the whole space $\mathbb{R}^n$ as the level goes to infinity. 
	Such a phenomena can occur when we work with the product $fg$ of some functions, as the  complement $\mathbb{R}^2\setminus
	{\rm Hess}^+(fg)$ might be bounded in many cases when $\mathbb{R}^2\setminus {\rm Hess}^+(f)$ and $\mathbb{R}^2\setminus
	{\rm Hess}^+(g)$ are bounded, as Theorem \ref{th05.03.2021.2} and Proposition \ref{prop05.03.2021.2} show. For example the  complement $\mathbb{R}^2\setminus {\rm Hess}^+(f_ag_a)$ is bounded, as it will be proved in Example \ref{ex:15.01.2022.1}.

	\section{Critical points and Hessian matrices}
	Since critical points appear repeatedly in our approach, we will recall them quickly. If $f:\mathbb{R}^n\longrightarrow\mathbb{R}^m$ is a Fr\' echet differentiable mapping, then the {\em rank}\label{r12} of $f$ at $x\in \mathbb{R}^n$ is defined as ${\rm rank}_xf:=\mbox{rank}(df)_x=\mbox{rank}(Jf)_x$. Observe that 
	${\rm rank}_xf\leq\min\{m,n\}$ for every $x\in \mathbb{R}^n$. A point $x\in \mathbb{R}^n$ is said to be a {\em critical point} of $f$ 
	if ${\rm rank}_xf<\min\{m,n\}$. Otherwise $x$ is said to be a {\em regular point} of $f$. If $f:\mathbb{R}^n\longrightarrow\mathbb{R}^m$ is a $C^1$-smooth map, then each point $x\in \mathbb{R}^n$ has an open neighbourhood, 
	say $V_x\subseteq \mathbb{R}^n$, such that ${\rm rank}_yf\geq {\rm rank}_xf$, for all $y\in V_x$. In particular, once a point $x$ is regular, it has a whole neighbourhood 
	of regular points. Indeed the Jacobian matrix $(Jf)_x$ has a non-zero minor of order ${\rm rank}_xf$ 
	and all minors of $(Jf)_x$ of superior order are zero. But the nonzero minor of $(Jf)_x$ is nonzero on a whole open neighbourhood of $x$ since it is a continuous function. 
	This shows that ${\rm rank}_yf={\rm rank}_y(Jf)_y\geq \mbox{rank}(Jf)_x={\rm rank}_xf$, which are satisfied for $y$ in a whole neighbourhood of $x$.
	Consequently the set $R(f)$, of regular points of $f$, is open in 
	$\mathbb{R}^n$, while the set $C(f)$, of critical points of $f$, is closed in $\mathbb{R}^n$. The set of critical values of $f$ is $B(f):=f(C(f))$. 
	Note that for a real valued function $f:U\longrightarrow\mathbb{R}$, the critical set of $f$ is the vanishing set of its gradient $\nabla f$. Since 
	$d_p(fg)=f(p)(dg)_p+g(p)(df)_p$ it follows immediately that \[\nabla_p(fg)=f(p)\nabla_pg+g(p)\nabla_pf,\]
	where $\nabla_ph$ stands for the gradient of $h:\mathbb{R}^n\longrightarrow\mathbb{R}$ at $p\in\mathbb{R}^n$, i.e. $\nabla_ph=(h_{x_1}(p),\ldots,h_{x_n}(p))$.
	\begin{remark}\label{rem20.04.2021.1}
		If $f:\mathbb{R}^n\longrightarrow\mathbb{R}$ is a $C^1$-smooth convex function, then its critical set $C(f)$ is convex. 
		Indeed the critical points of $f$ coincide with the global minimum points of $f$ (see \cite[Theorem 2.5.7]{Zalinescu}). 
		In other words \[C(f)=\{p\in\mathbb{R}^n \ : \ f(p)\leq f(x), \ \forall x\in\mathbb{R}^n\}.\]
		Note that $C(f)=f^{-1}(-\infty,f(p)]$ for every $p\in C(f)$, i.e. $C(f)$ is a sublevel set of the convex function $f$ in this particular case. Therefore 
		$C(f)$ is convex.
		Once we know \cite[Theorem 2.5.7]{Zalinescu}, we can directly prove the convexity of $C(f)$ in this particular case. 
		If $p,q\in C(f)$, then $d_pf=d_qf=0$ and $p,q$ are therefore global minima of $f$, i.e. $f(p)=f(q)\leq f(u)$ for all $u\in\mathbb{R}^n$. Thus $f((1-t)p+tq)\leq(1-t)f(p)+tf(q)=f(p)=f(q)\leq f(u)$, 
		for all $t\in[0,1]$ and all $z\in\mathbb{R}^n$. This shows that $(1-t)x+ty$ is a global minimum point of $f$ for every $t\in[0,1]$ and therefore 
		$(1-t)x+ty\in C(f)$ for every $t\in[0,1]$.
	\end{remark}
	\noindent By using Remark \ref{rem20.04.2021.1} we shall provide examples of nonconvex functions which have, or are good candidates to have, bounded ${\rm Hess}^+$ complements.
	\begin{example}\label{ex03.05.2021.1}
		\begin{enumerate}
			\item By solving the equation $\nabla f_a=0$ one obtains the discrete solution set $C(f_a)=\{(-a,0),(0,0),(a,0)\}$, 
			which is the critical set of the polynomial function 
			\[f_a:\mathbb{R}^2\longrightarrow\mathbb{R}, \ f_a(x,y)=(x^2+y^2)^2-2a^2(x^2-y^2), \ (a>0).\]
			Therefore $f_a$ is not a convex function via Remark \ref{rem20.04.2021.1}.

		\end{enumerate}
		\end{example}
	
		\begin{example}

Finally, the product $f_ag_a=(d_{(-a,0)}^2d_{(a,0)}^2-a^4)(d_{(0,-a)}^2d_{(0,a)}^2-a^4)$ ($a>0$) is not convex, as its critical set is 
		 \[
		 C(f_ag_a)=\left\{\left(\pm a\sqrt[4]{2},0\right),(0,0),\left(0,\pm a\sqrt[4]{2}\right)\right\}.
		 \]
		 Indeed the gradients of $f_a$ and $g_a$ are 
		 \[
		 \nabla_{(x,y)}f_a=4(x^2+y^2)(x,y)-4a^2(x,-y), \ \nabla_{(x,y)}g_a=4(x^2+y^2)(x,y)+4a^2(x,-y)
		 \]
		 and therefore
		 \begin{align}\label{eq11.01.2022.1}
		 (x,y)\in C(f_ag_a) & \Longleftrightarrow f_a(x,y)\nabla_{(x,y)}g_a+g_a(x,y)\nabla_{(x,y)}f_a=0\\
		 & \Longleftrightarrow 4(x^2+y^2)(f_a(x,y)+g_a(x,y))(x,y)+(4a^2f_a(x,y)-4a^2g_a(x,y))(x,-y)=(0,0).\nonumber
		 \end{align}
		 If $x\neq 0$ and $y\neq 0$, then the vectors $(x,y)$ and $(x,-y)$ are linearly independent and the equation \eqref{eq11.01.2022.1}
		 is equivalent with 
		 \[
		 \left\{
		 \begin{array}{ll}
		 4(x^2+y^2)(f_a(x,y)+g_a(x,y))=0\\
		 4a^2f_a(x,y)-4a^2g_a(x,y)=0 
		 \end{array}		 
		 \right.\Longleftrightarrow x=0\mbox{ and }y=0.
		 \] On the other hand, for $y=0$, the equation \eqref{eq11.01.2022.1} is equivalent with  
		\begin{align}\label{eq11.01.2022.2}
		4x^2(f_a &(x,0)+g_a(x,0))x+(4a^2f_a(x,0)-4a^2g_a(x,0))x=0\nonumber\\
		& \Longleftrightarrow x=0 \mbox{ or }2x^6-4a^4x^2=0\nonumber\\
		& \Longleftrightarrow x=0 \mbox{ or  }x=\pm\sqrt[4]{2}a.\nonumber
		\end{align}
		Similary, for $x=0$  the equation \eqref{eq11.01.2022.1} is equivalent with $y=0$ or $y=\pm \sqrt[4]{2}a.$
		In particular $C(f_ag_a)=$ and 
		$B(f_ag_a)=\left\{-4a^8,0\right\}$.
		\end{example}
		
	Throughout the paper we make use of the notation described below (see also \cite{Pintea-Trif}).
	Let $D$ be a nonempty open convex subset of $\mathbb{R}^n$, and let
	$f:D\to\mathbb{R}$ be a $C^2$-smooth function. The Hessian matrix of $f$ at an arbitrary point $x\in D$
	will be denoted by $H_x(f)$. Recall that $H_x(f)$ is a symmetric matrix and it defines a symmetric bilinear functional 
	\[
	{\mathcal H}_f(x):\mathbb{R}^n\times\mathbb{R}^n\longrightarrow\mathbb{R}, \ {\mathcal H}_f(x)(u,v):=u\cdot H_x(f)\cdot v^T.
	\]
	Denote by $h_f(x):\mathbb{R}^n\longrightarrow\mathbb{R}^n$ the linear transformation defined by the following equality 
	\[
	{\mathcal H}_f(x)(u,v):=\left\langle h_f(x)u,v\right\rangle, \ \forall u,v\in \mathbb{R}^n.
	\]
	In fact $h_f(x)u=u\cdot H_x(f)$. Note that the operator $h_f(x)$ is symmetric, i.e. 
	$\left\langle h_f(x)u,v\right\rangle=\left\langle u,h_f(x)v\right\rangle$,  for all $u,v\in \mathbb{R}^n$.

	Further, let $A:\mathbb{R}^n\to\mathbb{R}^n$ be a linear operator.,
	We shall denote by $[A]$ the matrix representation of $A$ with respect to the standard basis of $\mathbb{R}^n$.
	Let $S^{n-1}$ denote the unit sphere (i.e., centered at the origin) in $\mathbb{R}^n$, and let
	\begin{equation*}
		W(A):=\{ \langle Ax, x\rangle \mid x\in S^{n-1}\}
	\end{equation*}
	be the {\em numerical range}\/ of $A$. It is well known that
	$W(A)=[\lambda(A),\mu(A)]$, where $\lambda(A)$ and $\mu(A)$ denote the smallest and the greatest eigenvalue, respectively, of the symmetric operator $(A+A^*)/2$. Consequently $$\lambda(A)=\min_{x\in S^{n-1}}\langle Ax, x\rangle\mbox{ and }
	\mu(A)=\max_{x\in S^{n-1}}\langle Ax, x\rangle.$$
	In particular \[\lambda(H_p(f)):=\lambda(h_f(p))=\min_{u\in S^{n-1}}u\cdot H_p(f)\cdot u^T\]
	and 
	\[
	\mu(H_p(f)):=\mu(h_f(p))=\max_{u\in S^{n-1}}u\cdot H_p(f)\cdot u^T.
	\]
	Moreover, $H_p(f)$ is positive definite if and only if $\lambda(H_p(f))>0$. We are interested about the region 
	\[
	{\rm Hess}^+(f)=\{x\in D : H_x(f)\mbox{ is positive definite}\}.
	\]
	Note that for $n=2$ we have ${\rm Hess}^+(f)=\{x\in D : {\rm Tr}(H_x(f)),\ \det(H_x(f))>0\}$. Let us also recall from \cite{Pintea-Tofan} that 
	${\rm Hess}^+(f_a)=\{(x,y)\in\mathbb{R}^2 : 3(x^2+y^2)^2+2a^2(x^2-y^2)>a^4\}$.

	\begin{figure}[ht]
		\includegraphics[width=10cm]{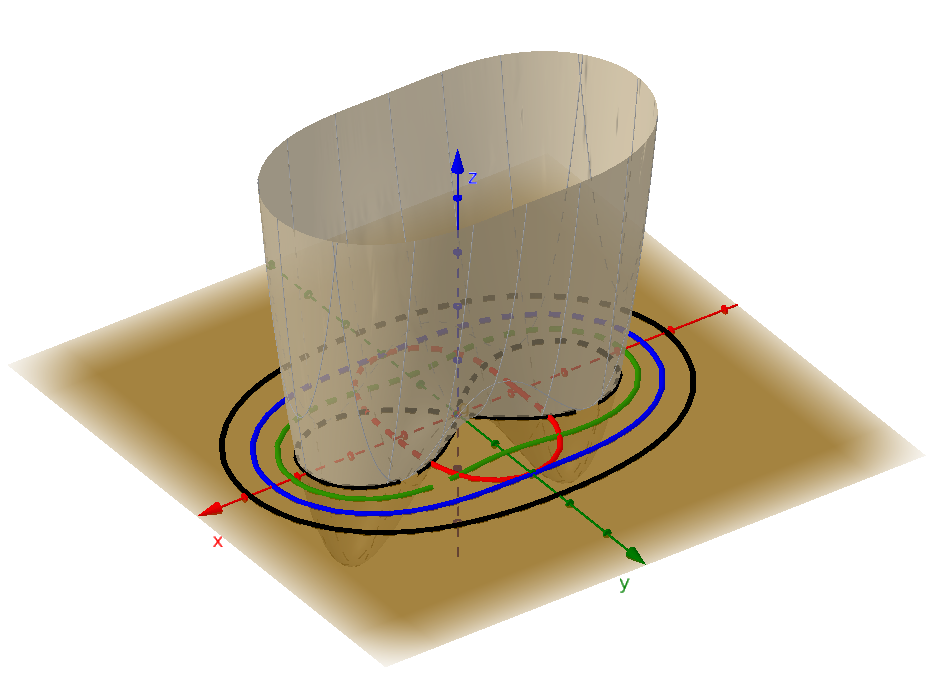}
		\caption{A piece of graph of $f_a$. \textcolor{red}{The curve $\partial{\rm Hess}^+(f_a)$}, the zero level of $f_a$-the Bernoulli's lemniscate, \textcolor{green}{a nonconvex regular level of $f_a$},
			\textcolor{blue}{the first positive convex regular level of $f_a$} and a regular convex level $f_a$ completely contained in ${\rm Hess}^+(f_a)$.}
		\label{Hess^+(f_a)}
	\end{figure}
	
		\noindent Once we have some boundedness information on the complements $\mathbb{R}^2\setminus {\rm Hess}^+(f)$ and $\mathbb{R}^2\setminus {\rm Hess}^+(g)$, where $f,g$ are two real valued functions, we will be looking for  boundedness information on the complement $\mathbb{R}^2\setminus {\rm Hess}^+(fg)$. In this respect we first recall the formulas for the gradient and the Hessian matrix of the product $fg$. In order to justify the formula for the Hessian matrix of the product we first recall the formula for the Jacobian matrix of 
	the exterior product $fF$, where $F$ is a vector-valued function. 
	\begin{align}
		& d_p(fF)(q) =f(p)(d_pF)(q)+(d_pf)(q)F(p)\nonumber\\
		& \Leftrightarrow J_p(fF)q^T =f(p)(J_pF)q^T+\langle \nabla_pf,q\rangle F(p)^T\nonumber\\
		& \Leftrightarrow J_p(fF)q^T =f(p)(J_pF)q^T+[F(p)^T\nabla_p f]q^T.\nonumber
	\end{align}
	\noindent Thus $J_p(fF)=f(p)(JF)_p+F(p)^T\nabla_p f$.
	\begin{align}
		H_p(fg) & =J_p(\nabla fg)=J_p(f\nabla g+g\nabla f)  \nonumber\\
		& =J_p(f\nabla g)+J_p(g\nabla f) \nonumber\\
		& =f(p)(J_p\nabla g+(\nabla_pg)^T\nabla_pf)
		+g(p)(J_p\nabla f)+(\nabla_pf)^T\nabla_pg\nonumber\\
		& = f(p)H_p(g)+g(p)H_p(f)+(\nabla_pf)^T\nabla_pg+(\nabla_pg)^T\nabla_pf.\label{Hessian-of-product}
	\end{align}

	Some other useful formulas are those for the gradient and the Hessian matrix of the composed function $\varphi\circ f$, where 
	$\varphi:\mathbb{R}\longrightarrow\mathbb{R}$ is a $C^2$-smooth function. One can easily check that 
	\begin{align}
		& \nabla(\varphi\circ f)=(\varphi'\circ f)\nabla f, \ \mbox{ i.e. }d(\varphi\circ f)=(\varphi'\circ f)df\label{eq10.04.2021.5}\\
		& H(\varphi\circ f)= (\varphi'\circ f)H(f)+(\varphi''\circ f)(\nabla f)^T\nabla f.\label{eq10.04.2021.2}
	\end{align}
	\begin{remark}
		 If $f>0$, then ${\rm 
        	 Hess}^+(f)\subseteq{\rm Hess}^+(f^2),$ 
        	 
        	 Indeed, if $H_p(f)$ is positive definite, then according to formula \eqref{Hessian-of-product}, the matrix \[H_p(f^2)=2f(p)H_p(f)+2(\nabla_p f)^T \nabla_p f\] is positive definite, as the sum of positive definite matrix $2f(p)H)p(f)$ and positive semidefinite matrix $2(\nabla_p f)^T\nabla_p f,$ with $\lambda(2(\nabla_p f)^T\nabla_p f)=0.$ 
            
            More generally, we can extend this result for other exponents, for     arbitrary $n,$
            Let $f>0$ such that $H_p(f)$ is positive definite everywhere.
            Then $H_p(f^2)$ is also positive everywhere and since, by formula \eqref{eq10.04.2021.5}, $\nabla_p f^{m+1}=(m+1) f^m(p) \nabla_p f$ and $$H_p(f^{m+1})=f(p) H_p(f^m)+f^m(p) H_p(f)+2mf^{m-1}(p)(\nabla_p f)^T\nabla_p f, $$
            we can deduce inductively that $H_p(f^{m+1})$ is positive definite if $H_p(f^m)$ is positive definite, as $(\nabla_p f)^T\nabla_p f$ is positive semidefinite (being symmetric and having nonnegative eigenvalues).
            Thus, the following chain of inclusions becomes obvious for $f>0:$
            $${\rm Hess^+}(f)\subseteq {\rm Hess^+}(f^2)\subseteq\cdots \subseteq {\rm Hess^+}(f^n)\subseteq \cdots.$$
	\end{remark}	
Lastly, for a pair of real-valued functions $f,g:\mathbb{R}^n \longrightarrow \mathbb{R},$ we define the function \[f\oplus g:\mathbb{R}^n\longrightarrow \mathbb{R}^2, (f\oplus g)(x)=(f(x),g(x)),\] which will prove very fruitful, as we'll see shortly.
Note that $f\oplus g$ is differentiable whenever $f$ and $g$ are differentiable and $d_p(f\oplus g)=d_pf\oplus d_pg$ 
	for every $p\in\mathbb{R}^n$.  Note that $f\oplus g$ and  $g\oplus f$ are usually not equal. In fact  $f\oplus g=r\circ(g\oplus f)$, where $r_b:\mathbb{R}^2\longrightarrow\mathbb{R}^2$ stands for the reflection about the first bisector $b: x=y$.

	\begin{remark}\label{rem11.04.2021.1}
		Let $f,g:\mathbb{R}^n\longrightarrow\mathbb{R}$ be $C^2$-smooth functions.
		\begin{enumerate}
			\item If $f,g\geq 0$, then
			\[\lambda(H_p(fg))\geq f(p)\lambda(H_p(g))+g(p)\lambda(H_p(f))+\lambda((\nabla_pf)^T\nabla_pg+(\nabla_pg)^T\nabla_pf).\]
			Indeed, by using the obvious fact that $\langle Au,u\rangle=uAu^T\geq\lambda(A)\|u\|^2$ for all vectors $u$ and the formula \eqref{Hessian-of-product} 
			for the Hessian of the product $fg$ we observe that
			\begin{align}
				& uH_p(fg)u^T =u[f(p)H_p(g)+g(p)H_p(f)+(\nabla_pf)^T\nabla_pg+(\nabla_pg)^T\nabla_pf]u^T\nonumber\\
				& =f(p)uH_p(g)u^T+g(p)uH_p(f)u^T+u[(\nabla_pf)^T\nabla_pg+(\nabla_pg)^T\nabla_pf]u^T\nonumber\\
				& \geq [f(p)\lambda(H_p(g))+g(p)\lambda(H_p(f))+\lambda((\nabla_pf)^T\nabla_pg+(\nabla_pg)^T\nabla_pf)]\|u\|^2.\nonumber
			\end{align}
			Thus 
			\[uH_p(fg)u^T\geq f(p)\lambda(H_p(f))+g(p)\lambda(H_p(g))+\lambda((\nabla_pf)^T\nabla_pg+(\nabla_pg)^T\nabla_pf), \ 
			\] 
			for all $u\in S^{n-1}$, which shows that 
			\begin{align}
				\lambda(H_p(fg)) & =\min_{u\in S^{n-1}}uH_p(fg)u^T\nonumber
				& \geq f(p)\lambda(H_p(f))+g(p)\lambda(H_p(g))+\lambda((\nabla_pf)^T\nabla_pg+(\nabla_pg)^T\nabla_pf).\nonumber 
			\end{align}
			Therefore, the inequality
			\[
			f(p)\lambda(H_p(g))+g(p)\lambda(H_p(f))+\lambda((\nabla_pf)^T\nabla_pg+(\nabla_pg)^T\nabla_pf)>0 
			\]
			implies that $H_p(fg)$ is positive definite.
			\item $\lambda(H(\varphi\circ f))\geq (\varphi'\circ f)\lambda(H(f))$ whenever $\varphi:\mathbb{R}\longrightarrow\mathbb{R}$ is a convex increasing $C^2$-smooth function. 
			Indeed, by using the formula \eqref{eq10.04.2021.2} for the Hessian matrix of the composed function $\varphi\circ f$ we observe that
			\begin{align}
				\!\!\!\!\!\!\!\!\! & uH_p(\varphi\circ f)u^T =u[(\varphi'\circ f)(p)H_p(f)+(\varphi''\circ f)(p)(\nabla_pf)^T\nabla_pf]u^T\nonumber\\
				& =(\varphi'\circ f)(p)uH_p(f)u^T+(\varphi''\circ f)(p)u[(\nabla_pf)^T\nabla_pf]u^T\nonumber\\
				& \geq [(\varphi'\circ f)(p)\lambda(H_p(f))+(\varphi''\circ f)(p)\lambda((\nabla_pf)^T\nabla_pf)]\|u\|^2.\label{eq10.04.2021.1}
			\end{align}
			Note that $\lambda((\nabla_pf)^T\nabla_pf))=0$ and $\mu[(\nabla_pf)^T\nabla_pf)]=\|\nabla f\|^2$, as  the characteristic polynomial of $(\nabla_pf)^T\nabla_pf$ is 
			\[X^n-\|\nabla f\|^2X^{n-1}.\] Since the smallest eigenvalue of $(\nabla_pf)^T\nabla_pf$ is zero i.e. the symmetric matrix $(\nabla_pf)^T\nabla_pf$ is positive semi-definite, this implies that
			\[
			uH_p(\varphi\circ f)u^T\geq  (\varphi'\circ f)(p)\lambda(H_p(f)),
			\] 
			for all $u\in S^{n-1}$, which shows that 
			\begin{align}
				\lambda(H_p(\varphi\circ f)) & =\min_{p\in S^{n-1}}u_p H(\varphi\circ f)u^T \geq (\varphi'\circ f)(p)\lambda(H_p(f)).\nonumber 
			\end{align}
			Therefore, the inequality $(\varphi'\circ f)(p)\lambda[H_p(f)]>0$
			implies that $H_p(\varphi\circ f)$ is positive definite. This is the case when $\varphi'>0$ and $H_p(f)$ is positive definite.
			
		    		\item For two $C^1$-smooth functions $\varphi, \ \psi:\mathbb{R}\longrightarrow\mathbb{R},$ taking into account the formulas \eqref{eq10.04.2021.5}, one obtains that 
			\begin{equation}\label{eq16.04.2021.1}
				{\rm rank}[d(\varphi\circ f)\oplus d(\psi\circ f)]\le 1, \mbox{ i.e. }C[(\varphi\circ f)\oplus (\psi\circ f)]=\mathbb{R}^2.
			\end{equation}
			Indeed, ${\rm rank}[(\varphi'\circ f)df\oplus (\psi'\circ f)df]\le {\rm rank}(df\oplus df)\le 1$. 
			
			In particular, when $n=2$ we can see that $	\det[d(\varphi\circ f)\oplus d(\psi\circ f)]=0.$
			
			Indeed, $\det[(\varphi'\circ f)df\oplus (\psi'\circ f)df]=(\varphi'\circ f)(\psi'\circ f)\det(df\oplus df)=0$.
			\item Since $(\nabla_p f)^T \nabla_p g$ and $(\nabla_p g)^T \nabla_p f$ are matrices of rank $1,$ the symmetric matrix $(\nabla_pf)^T\nabla_pg+
				(\nabla_pg)^T\nabla_pf $ has a rank of at most $2.$ Furthermore, since \[{\rm Tr}((\nabla f)^T\nabla g+
				(\nabla g)^T\nabla f) =2\left<\nabla f, \nabla g\right>\] \[M:=\sigma_2((\nabla f)^T\nabla g+
				(\nabla g^T\nabla f) =-\sum_{1\le i<j\le n} (f_{x_i}g_{x_j}-f_{x_j}g_{x_i})^2, \]
			
			by solving the characteristic equation $X^n-2\left<\nabla f, \nabla g\right>X^{n-1}+M X^{n-2}=0,$ one finds that 
			\begin{align}
				\lambda_{1,2}
				& =\langle \nabla f,\nabla g\rangle\pm
				\sqrt{\langle\nabla f,\nabla g\rangle^2+
					\sum_{1\le i<j\le n} (f_{x_i}g_{x_j}-f_{x_j}g_{x_i})^2}\nonumber\\
				& =\langle \nabla f,\nabla g\rangle \pm\|\nabla f\|\cdot\|\nabla g\|,  \nonumber
			\end{align} are the only potentially nonzero eigenvalues of $(\nabla f)^T\nabla g+(\nabla g)^T\nabla f $, by Lagrange's Identity.
			
			Evidently, $\lambda((\nabla_pf)^T\nabla_pg+(\nabla_pg)^T\nabla_pf)=\left<\nabla_p f,\nabla_p g\right>-\|\nabla_p f\|\cdot \|\nabla_p g\|\le 0$ and $\mu((\nabla_pf)^T\nabla_pg+(\nabla_pg)^T\nabla_pf)=\left<\nabla_p f,\nabla_p g\right>+\|\nabla_p f\|\cdot \|\nabla_p g\|\ge 0,$ by the Cauchy-Schwartz inequality. 
			
			Note that $\lambda((\nabla_pf)^T\nabla_pg+(\nabla_pg)^T\nabla_pf)=0$ implies that rank $d_p(f\oplus g)\le 1$, i.e. $p\in C(f\oplus g)$.
			Conversely, if $p\in C(f\oplus g)$ and $\langle\nabla_pf,\nabla_pg\rangle\geq 0$, then
			\[\lambda((\nabla_pf)^T\nabla_pg+(\nabla_pg)^T\nabla_pf)=0.\] However \[
			\lambda((\nabla_pf)^T\nabla_pg+(\nabla_pg)^T\nabla_pf)=2\langle\nabla_pf,\nabla_pg\rangle
			\]
			whenever $p\in C(f\oplus g)$ and $\langle\nabla_pf,\nabla_pg\rangle\leq 0$.
			Thus \[\lambda((\nabla_pf)^T\nabla_pg+(\nabla_pg)^T\nabla_pf)=0\] whenever $p\in C(f)\cup C(g)$ as 
			\[\left<\nabla_pf,\nabla_pg\right>=\|\nabla_pf\|\cdot \|\nabla_pg\|=0\] for such $p$.

		\end{enumerate}
	\end{remark}
	\section{The main results}
	In this section we first provide a source of examples of polynomial functions with large ${\rm Hess}^+$ region, i.e. the complement of this region is compact. We provide some sufficient conditions on two functions $f,g:\mathbb{R}^2\longrightarrow\mathbb{R}$ with  bounded ${\rm Hess}^+$ complements whose product $fg$ keeps having  bounded ${\rm Hess}^+$ complement as well.
	\begin{theorem}\label{th05.03.2021.1}
		If $f:\mathbb{R}^2\longrightarrow\mathbb{R}$, $f(x,y)=P(x^2+y^2)+p(x,y)$, where $P(z)\in\mathbb{R}[z]$ is a 
		polynomial function with non-negative coefficients and $p\in\mathbb{R}[x,y]$ is such that $2\deg(P)>\deg(p)\geq 2$, 
		then $\mathbb{R}^2\setminus{\rm Hess}^+(f)$ is bounded.
	\end{theorem}
	\begin{proof}
		The Hessian matrix $H_{(x,y)}(f)$ of $f$ is 
		\[
		\left(
		\begin{array}{cc}
		2P'(x^2+y^2)+4x^2P''(x^2+y^2)+p_{xx} & 4xyP''(x^2+y^2)+p_{xy}\\
		4xyP''(x^2+y^2)+p_{xy} & 2P'(x^2+y^2)+4y^2P''(x^2+y^2)+p_{yy} 
		\end{array}
		\right)
		\]

 and it is positive definite for $\|(x,y)\|$ sufficiently large, 

seeing as 
		\[
		\lim_{\|(x,y)\|\rightarrow+\infty}\Delta f=+\infty\mbox{ and }\lim_{\|(x,y)\|\rightarrow+\infty}\det(H_{(x,y)}(f))=+\infty.
		\]
		Indeed,
		\begin{align}
			& \lim_{\|(x,y)\|\rightarrow+\infty}\Delta f = \lim_{\|(x,y)\|\rightarrow+\infty}(4P'(x^2+y^2)+4(x^2+y^2)P''(x^2+y^2)+\Delta p)\nonumber\\
			& = \lim_{\|(x,y)\|\rightarrow+\infty}(x^2+y^2)^{n-1}\left(4\cfrac{P'(x^2+y^2)}{(x^2+y^2)^{n-1}}+
			4\cfrac{P''(x^2+y^2)}{(x^2+y^2)^{n-2}}+\cfrac{\Delta p}{(x^2+y^2)^{n-1}}\right)\nonumber\\
			& =+\infty,\nonumber
		\end{align}
		as
		\[
		\lim_{\|(x,y)\|\rightarrow+\infty}
		\left(4\cfrac{P'(x^2+y^2)}{(x^2+y^2)^{n-1}}+
		4\cfrac{P''(x^2+y^2)}{(x^2+y^2)^{n-2}}+\cfrac{\Delta p}{(x^2+y^2)^{n-1}}\right)=4n^2a_0>0,
		\]
		where $n=\deg(P)$ and $a_0$ is the coefficient of $z^n$ in $P(z)$, respectively.
		\begin{align}
			& \lim_{\|(x,y)\|\rightarrow+\infty}\det(H_{(x,y)}(f))\nonumber\\
			& = \lim_{\|(x,y)\|\rightarrow+\infty}\Big(4(P'(x^2+y^2))^2+8(x^2+y^2)P'(x^2+y^2)P''(x^2+y^2)\nonumber\\
			& + 2(\Delta p)P'(x^2+y^2)+\det H_{(x,y)}(p)+4(x^2p_{yy}-2xyp_{xy}+y^2p_{xx})P''(x^2+y^2)\Big)\nonumber\\
			& = \lim_{\|(x,y)\|\rightarrow+\infty}(x^2+y^2)^{2n-2}\Bigg(4\left(\cfrac{P'(x^2+y^2)}{(x^2+y^2)^{n-1}}\right)^2+
			8\cfrac{P'(x^2+y^2)P''(x^2+y^2)}{(x^2+y^2)^{2n-3}}\nonumber\\
			& + 2\cfrac{(\Delta p)P'(x^2+y^2)}{(x^2+y^2)^{2n-2}}+
			\cfrac{\det H_{(x,y)}(p)}{(x^2+y^2)^{2n-2}} +4\cfrac{x^2p_{yy}-2xyp_{xy}+
				y^2p_{xx}}{(x^2+y^2)^n}\cfrac{P''(x^2+y^2)}{(x^2+y^2)^{n-2}}\Bigg)\nonumber\\
			& =+\infty,\nonumber
		\end{align}
		as
		\begin{align}
			& \lim_{\|(x,y)\|\rightarrow+\infty}\Bigg(4\left(\cfrac{P'(x^2+y^2)}{(x^2+y^2)^{n-1}}
			\right)^2+
			8\cfrac{P'(x^2+y^2)P''(x^2+y^2)}{(x^2+y^2)^{2n-3}}\nonumber\\
			& + 2\cfrac{(\Delta p)P'(x^2+y^2)}{(x^2+y^2)^{2n-2}}+
			\cfrac{\det H_{(x,y)}(p)}{(x^2+y^2)^{2n-2}} +4\cfrac{x^2p_{yy}-2xyp_{xy}+
				y^2p_{xx}}{(x^2+y^2)^n}\cfrac{P''(x^2+y^2)}{(x^2+y^2)^{n-2}}\Bigg)\nonumber\\
			& =4n^2a_0^2+8n^2(n-1)a_0^2=4n^2(2n-1)a_0^2>0,\nonumber
		\end{align}
		and the proof is now complete.
	\end{proof}
	\begin{corollary}\label{cor17.04.2021.1}
		If $f,g:\mathbb{R}^2\longrightarrow\mathbb{R}$, are defined by $f(x,y)=P(x^2+y^2)+p(x,y)$ and  $g(x,y)=Q(x^2+y^2)+q(x,y)$,
		where $P(z),Q(z)\in\mathbb{R}[z]$ are polynomial functions with non-negative coefficients and $p,q\in\mathbb{R}[x,y]$
		are such that $2\deg(P)>\deg(p)\geq 2$ and  $2\deg(Q)>\deg(q)\geq 2$,
		then $\mathbb{R}^2\setminus{\rm Hess}^+(fg)$ is bounded.
	\end{corollary}
	\begin{proof}
		We only need to observe that the product is a polynomial function of the same type with $f$ and $g$. Indeed,
		\begin{align}
			fg & =[P(x^2+y^2)+p(x,y)][Q(x^2+y^2)+q(x,y)]\nonumber\\
			& =(PQ)(x^2+y^2)+P(x^2+y^2)q(x,y)+Q(x^2+y^2)p(x,y)+p(x,y)q(x,y)\nonumber
		\end{align}
		and $PQ$ have nonnegative coefficients as well as 
		\begin{align}
			2\deg(PQ) & =2\deg P+2\deg Q \nonumber\\
			& > \max\{\deg P(x^2+y^2)q(x,y), \deg Q(x^2+y^2)p(x,y),\deg p(x,y)q(x,y)\}\nonumber\\
			& =\max\{2\deg P+\deg q, 2\deg Q+\deg p,\deg p+\deg q\}\geq 2,\nonumber
		\end{align}
		and the proof of the statement is now complete.
	\end{proof}
	\begin{remark}
		If $P(z)=z^2$ and $p(x,y)=2a^2(y^2-x^2)$, we deduce that the complement $\mathbb{R}^2\setminus{\rm Hess}^+(f_a)$ is bounded, where $f_a(x,y)=(x^2+y^2)^2-2a^2(x^2-y^2)$. Also $\mathbb{R}^2\setminus{\rm Hess}^+(g_b)$ is bounded, where 
		$g_b(x,y)=(x^2+y^2)^2+2b^2(x^2-y^2)$. In fact ${\rm Hess}^+(f_a)$ is precisely described in \cite{Pintea-Tofan}. According to Corollary \ref{cor17.04.2021.1}, the set $\mathbb{R}^2\setminus
		{\rm Hess}^+(f_ag_b)$ is bounded. 
		\begin{figure}[ht]
			\includegraphics[width=10cm]{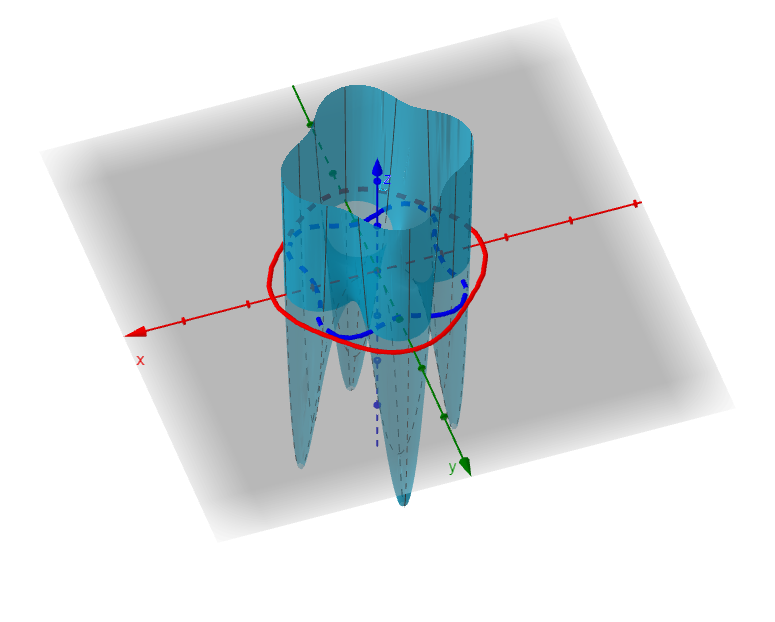}
			\caption{A piece of the graph of $f_1g_1$. \textcolor{blue}{The curve $\partial{\rm Hess}^+(f_1g_1)=h^{-1}\{0\}$}
				and \textcolor{red}{the first convex level set$(f_1g_1)^{-1}(d)$ completely contained in ${\rm Hess}^+(f_1g_1)$}}
			\label{Graph-of-a-product}
		\end{figure}
	\end{remark}
	\begin{theorem}\label{th05.03.2021.2}
	Let $f,g:\mathbb{R}^n\longrightarrow\mathbb{R}$ be two $C^2$-smooth functions such that 
		$\displaystyle\lim_{p\rightarrow\infty}f(p)=+\infty$, $\displaystyle\lim_{p\rightarrow\infty}g(p)=+\infty$ and $\langle \nabla f,\nabla g\rangle+\|\nabla f\|\cdot\|\nabla g\|>0$ almost everywhere.
		If the  sets \[R(f\oplus g), \ \mathbb{R}^n\setminus{\rm Hess}^+(f)\mbox{ and }\mathbb{R}^n\setminus{\rm Hess}^+(g)\]
		are additionally bounded, then $\mathbb{R}^n\setminus{\rm Hess}^+(fg)$ is bounded as well.
	\end{theorem}
	\begin{proof} Consider a ball $B=B(0,r)$ such that 
		$R(f\oplus g)=\mathbb{R}^n\setminus C(f\oplus g), \ \mathbb{R}^n\setminus{\rm Hess}^+(f), \ \mathbb{R}^n\setminus{\rm Hess}^+(g)\subseteq B$, i.e. 
		$\mathbb{R}^n\setminus B\subseteq{\rm Hess}^+(f)\cap {\rm Hess}^+(g)\cap  C(f\oplus g)$. 
		By increasing the radius $r$ one can also assume that
		$f\big|_{\mathbb{R}^n\setminus B}, \  g\big|_{\mathbb{R}^n\setminus B}>0$ and the inclusion $\mathbb{R}^n\setminus B\subseteq C(f\oplus g)$ shows that ${\rm rank} d(f\oplus g)\le 1$ on $\mathbb{R}^n\setminus B$, i.e. the linear 
		dependence of $\nabla f$ and $\nabla g$  on $\mathbb{R}^n\setminus B$. It can be easily seen that $\mathbb{R}^n\setminus B$ is the union of the sets
		\[
		(C(f)\cup C(g))\cap (\mathbb{R}^n\setminus B)\mbox{ and } (R(f)\cap R(g))\cap (\mathbb{R}^n\setminus B).
		\]
		If $p\in  (R(f)\cap R(g))\cap (\mathbb{R}^n\setminus B)$, i.e. none of the gradients  $\nabla_pf$ or $\nabla_pg$ vanishes, but still they are linearly dependent. Their linear dependence 
		implies the existence of a function $\alpha:\mathbb{R}^n\setminus B\longrightarrow\mathbb{R}^*$ such that $\nabla f=\alpha \nabla g$
		on $\mathbb{R}^n\setminus B$. Since $\alpha$ can be expressed through at leas one of the quotients \[\cfrac{f_{x_1}}{g_{x_1}}, \ \cfrac{f_{x_2}}{g_{x_2}},\ldots,\cfrac{f_{x_n}}{g_{x_n}}\] 
		whose denominator is not zero, it follows that the restriction of $\alpha$ to $(R(f)\cap R(g))\cap (\mathbb{R}^n\setminus B)$ is continuous. Since $\langle \nabla f,\nabla g\rangle+\|\nabla f\|\cdot\|\nabla g\|>0$ almost everywhere on $(R(f)\cap R(g))\cap (\mathbb{R}^n\setminus B)$
		it follows that $\alpha>0$ almost everywhere on $(R(f)\cap R(g))\cap (\mathbb{R}^n\setminus B)$ which shows that $\langle \nabla f,\nabla g\rangle=\alpha\|\nabla g\|^2>0$ 
		almost everywhere on $(R(f)\cap R(g))\cap (\mathbb{R}^n\setminus B)$. 
	
		Therefore
		\begin{equation}\label{eq11.05.2021.1}
			\lambda((\nabla f)^T\nabla g+(\nabla g)^T\nabla f)
			=\lambda(2\alpha (\nabla g)^T \nabla g) =0
		\end{equation}
		on $\mathbb{R}^n\setminus B$,
		as the equality \eqref{eq11.05.2021.1}
		on $(C(f)\cup C(g))\cap (\mathbb{R}^n\setminus B)$
		is obvious. Therefore 
		\begin{align}
			\lambda(H(fg)) & \geq f\lambda[H(g)]+g\lambda[H(f)]+\lambda((\nabla f)^T\nabla g+(\nabla g)^T\nabla f)\nonumber\\
			& = f\lambda[H(g)]+g\lambda[H(f)]>0\mbox{ on }\mathbb{R}^n\setminus B.\nonumber
		\end{align}
		and the statement is now completely proven.
	\end{proof}
	An extreme situation for the requirement on $R(f\oplus g)$ to be bounded, in Theorem \ref{th05.03.2021.2},
	is realized when the critical set $C(f\oplus g)$ is the whole space $\mathbb{R}^n$, i.e. ${\rm rank}~d(f\oplus g)\leq 1$. 
	This is the case when $f=g$, as ${\rm rank}~d(f\oplus f)=1$ and Theorem \ref{th05.03.2021.2} provides information on the set ${\rm Hess}^+(f^2)$.
	The collection of pairs $(f,g)$ for which ${\rm rank}~d(f\oplus g)=0$ can be slightly extended from pairs of type $(f,f)$ to pairs of type 
	$(\varphi\circ f,\psi\circ f)$, where $\varphi,\psi:\mathbb{R}\longrightarrow\mathbb{R}$ are strictly convex increasing differentiable functions.
	\begin{example}
		If $f:\mathbb{R}^n\longrightarrow\mathbb{R}$ is a $C^1$-smooth function then $i_{S^1}\circ {\rm exp}\circ f:\mathbb{R}^n\longrightarrow\mathbb{R}^2$, 
		where $S^1:=\{x\in\mathbb{R}^2 \ : \ \|x\|=1\}$ is the unit circle, $i_{S^1}:S^1\hookrightarrow\mathbb{R}^2$ is the inclusion map and
		${\rm exp}:\mathbb{R}\longrightarrow S^1$, ${\rm exp}(x)=e^{ix}$ is the exponential function,
		is a $C^1$-smooth function and $C(i_{S^1}\circ {\rm exp}\circ f)=\mathbb{R}^n$, i.e.  $R(i_{S^1}\circ {\rm exp}\circ f)=\emptyset$.
	\end{example}
	\begin{proposition}\label{prop16.04.2021.1}
		If $\varphi:\mathbb{R}\longrightarrow\mathbb{R}$ is a convex increasing $C^2$-smooth function and $f:\mathbb{R}^n\longrightarrow\mathbb{R}$
		is a $C^2$-smooth function, then ${\rm Hess}_0^+(f)\subseteq {\rm Hess}_0^+(\varphi\circ f)$, where ${\rm Hess}_0^+(F)$ 
		stands for the region where $H_F$ is positive semi-definite. Moreover the following inclusion holds 
		${\rm Hess}^+(f)\subseteq {\rm Hess}^+(\varphi\circ f)$, whenever $\varphi'>0$.
	\end{proposition}
	\begin{proof}
		
		The characteristic polynomial of the symmetric matrix $(\nabla f)^T\nabla f$ is $X^n-\|\nabla f\|^2X^{n-1}$ 
		and its smallest eigenvalue is therefore zero. Thus, the symmetric matrix $(\nabla f)^T\nabla f$ is everywhere 
		positive semi-definite and by multiplying it with the non-negative function $\varphi''\circ f$ we get a new positive semi-definite 
		matrix. On the other hand the Hessian matrix $H(f)$ is positive semi-definite at every point of ${\rm Hess}_0^+(f)$ and  by multiplying it with
		the non-negative function $\varphi'\circ f$ we get a new matrix $(\varphi'\circ f)H(f)$ which is positive semi-definite everywhere on
		${\rm Hess}_0^+(f)$. Consequently the symmetric matrix
		$H_{\varphi\circ f}= (\varphi'\circ f)H_f+(\varphi''\circ f)(\nabla f)^T\nabla f$ is 
		positive semidefinite everywhere on ${\rm Hess}_0^+(f)$, i.e. ${\rm Hess}_0^+(f)\subseteq {\rm Hess}_0^+(\varphi\circ f)$.
		Finally, by multiplying the Hessian matrix $H(f)$ with the positive function $\varphi'\circ f$ 
		we get a new matrix $(\varphi'\circ f)H(f)$ which is positive definite everywhere on ${\rm Hess}^+(f)$, 
		whenever $\varphi'>0$. Therefore the symmetric matrix
		\[H(\varphi\circ f)= (\varphi'\circ f)H(f)+(\varphi''\circ f)(\nabla f)^T\nabla f\] is 
		positive definite everywhere on ${\rm Hess}^+(f)$, i.e. ${\rm Hess}^+(f)\subseteq {\rm Hess}^+(\varphi\circ f)$, whenever $\varphi'>0$.
	\end{proof}
\begin{proposition}\label{prop05.03.2021.2}
		Let $f:\mathbb{R}^n\longrightarrow\mathbb{R}$ be a $C^2$-smooth function with zero measure critical set and
		$\displaystyle\lim_{p\rightarrow\infty}f(p)=+\infty$. Let also $\varphi, \psi:\mathbb{R}\longrightarrow\mathbb{R}$ be 
		convex unbounded functions such that $\varphi',\psi'>0$.
		If $\mathbb{R}^n\setminus{\rm Hess}^+(f)$ is bounded, 
		then \[\mathbb{R}^n\setminus{\rm Hess}^+(\varphi\circ f)(\psi\circ f)\] is bounded as well. 
	\end{proposition}
	\begin{proof} We first observe that, under the hypothesis on the functions $\varphi$ and $\psi$, we have
		\begin{align}
			\langle \nabla (\varphi\circ f),\nabla(\psi\circ f)\rangle+\|\nabla (\varphi\circ f)\|\cdot\|\nabla(\psi\circ f)\| =  
			2(\varphi'\circ f)(\psi'\circ f)\|\nabla f\|^2>0,\nonumber
		\end{align}
		almost everywhere, as $\varphi'\circ f, \ \psi'\circ f>0$ and $\|\nabla f\|^2>0$ almost everywhere. By using Proposition 
		\label{prop17.04.2021.9}, one can deduce that ${\rm Hess}^+(f)\subseteq {\rm Hess}^+(\varphi\circ f)\cap {\rm Hess}^+(\psi\circ f)$,
		i.e.  
		\[
		\mathbb{R}^n\setminus{\rm Hess}^+(\varphi\circ f), \ \mathbb{R}^n\setminus{\rm Hess}^+(\psi\circ f)\subseteq
		\mathbb{R}^n\setminus{\rm Hess}^+(f),
		\]
		which shows that $\mathbb{R}^n\setminus{\rm Hess}^+(\varphi\circ f), \ \mathbb{R}^n\setminus{\rm Hess}^+(\psi\circ f)$ 
		are both bounded, as $\mathbb{R}^2\setminus{\rm Hess}^+(f)$ is bounded by hypothesis. 
		On the other hand \[{\rm rank}[d(\varphi\circ f)\oplus d(\psi\circ f)]\le 1\] due to result \eqref{eq16.04.2021.1}. Therefore the functions $\varphi\circ f$ and $\psi\circ f$ 
		satisfy the requirements of Theorem \ref{th05.03.2021.2}
		which shows that $\mathbb{R}^n\setminus{\rm Hess}^+(\varphi\circ f)(\psi\circ f)$ is bounded.
	\end{proof}
	In the particular example we are going to analyse in the next section the level sets will be compact as the involved functions are norm-coercive.	Moreover the critical sets were finite and therefore, above its greatest critical value, the level sets are also regular. 
	Finally, the level curves above a certain regular level are actually, not only regular, but also convex. 
	In the last theorem of this section we shall show that this is actually the case whenever the function is norm-coercive 
	and the critical set along with the complement of the ${\rm Hess}^+$ region are both compact.
	\begin{theorem}\label{th:05.01.2022.1}
	Let $f:\mathbb{R}^2\longrightarrow\mathbb{R}$ be a $C^2$-smooth norm-coercive function. If $C(f)$ and $\mathbb{R}^2\setminus {\rm Hess}^+(f)$ are bounded, then the level $f^{-1}(c)$ is compact connected regular and convex for $c$ sufficiently large. 
	\end{theorem}
	\begin{remark}\label{rem:15.01.2022.1}
	To select the convex sublevel sets of a certain function we rely on \cite[Proposition 1, p. 397]{Carmo} combined \cite[Formula 3.7]{Goldman}, by making also sure that such a sublevel set is a priori connected. In this respect we need to first 
	select those levels of $f$ along which the determinant 
	\[
	D(f):=\left|
\begin{array}{lll}
f_{xx} & f_{xy} & f_x\\
f_{yx} & f_{yy} & f_y\\
f_x & f_y & 0
\end{array}
\right|
	\]
	keeps constant sign and select, among these level sets those which are also connected. Note that the restriction of $ D(f)$
to the region ${\rm Hess}^+(f)$ is negative, as
\begin{align}
D(f) & =2f_xf_yf_{xy}-f_x^2f_{yy}-f_y^2f_{xx} =-(-f_y,f_x)H(f)(-f_y,f_x)^T\nonumber\\
& =-\mathcal{H}_f((r_y\circ r_b)(\nabla f),(r_y\circ r_b)(\nabla f))<0\mbox{ over }{\rm Hess}^+(f),\nonumber
\end{align}
where $r_b,r_y:\mathbb{R}^2\longrightarrow\mathbb{R}^2$ are the reflections in the first bisector $b$ and in the $y$-axis respectively.
	\end{remark}
	\begin{proof}[Proof of Theorem \ref{th:05.01.2022.1}]
	 Since $C(f)$ and $\mathbb{R}^2\setminus {\rm Hess}^+(f)$ are both closed, the boundedness additional assumption on them assures their compactness.
	 We shall show that for $c>c_0:=\max_{x\in B}f(x)$ the level curve $f^{-1}(c)$ is compact connected, regular and convex, where $B$ is a closed ball 
	 which contains both the critical set and $\mathbb{R}^2\setminus {\rm Hess}^+(f)$. The compactness of all level sets of $f$ follows from the norm-coercivity of $f$. 
	 For $c>c_0:=\max_{x\in B}f(x)$ the level curve $f^{-1}(c)$ is obviously regular, as it contains no critical points and has finitely many connected components as
	 it is also compact. Let us also observe that the closed ball $B$ is contained in the interior

	 of every connected component of $f^{-1}(c)$, as $B$ has no common points with $f^{-1}(c)$ and cannot be contained in the exterior of any 
	 connected component $C$ of $f^{-1}(c)$. Indeed the lack of critical points in the interior of a certain $C$ is impossible as the restriction $f|_{{\rm cl ~ int}(C)}$ 
	 of $f$ to the compact set ${\rm cl ~ int}(C)$ has at least a minimum point and a maximum point, one of which does not belong to the boundary $C$ of ${\rm cl ~ int}(C)$ but to its interior ${\rm int}(C)$. 
	 This shows that the minimum points of this restriction are actually critical points of $f$. We therefore showed that the interior of every connected component 
	 of $f^{-1}(c)$ contains critical points of $f$ and with them these interiors contain the whole ball $B$, as $B$ is convex and has no common points with any of them.
	 Therefore the critical set $C(f)$, the ball $B$ and the connected components $C_1,\ldots,C_k$ of $f^{-1}(c)$, with a suitable indexing, are arranged as
	 \[
	  C(f)\subseteq B\subset {\rm int}(C_1)\subset\cdots\subset{\rm int}(C_k).
	 \]
	 Since the compact manifolds with boundary
	 \begin{equation}\label{eq:05.01.2022.1}
	 {\rm cl ~ int}(C_{i+1})\setminus{\rm int}(C_i), \ 1\leq i\leq k-1
	 \end{equation}
	 are compact and $f$ is the constant $c$ on the two components of each of their boundary, 
	 it follows that the minimum points of the restrictions of $f$ to the compact sets \eqref{eq:05.01.2022.1} belong to their interiors 
	 ${\rm int}(C_{i+1})\setminus{\rm cl ~ int}(C_i), \ 1\leq i\leq k-1$ and are therefore  critical points of $f$. But this is a contradiction with the relation
	 $C(f)\subseteq {\rm int}(C_1)$.
	  For $c>h_{\max}(f):=\max\{f(x) : x\in \mathbb{R}^2\setminus {\rm Hess}^+(f)\}$ we also have $\mathbb{R}^2\setminus {\rm Hess}^+(f)\subseteq f^{-1}(-\infty,c)\Longleftrightarrow f^{-1}[c,+\infty)\subseteq {\rm Hess}^+(f),$
	 i.e. all level sets $f^{-1}(c)$ are contained in ${\rm Hess}^+(f)$ for $c>h_{\max}(f)$ and they are convex, as the restriction 
	 of $D(f)$ to  due to Remark \ref{rem:15.01.2022.1}. 
	\end{proof}
	\begin{corollary}
	 Every two level curves of a function $f$ subject to the hypothesis of Theorem \ref{th:05.01.2022.1}, above the level $\mu_{\max}(f):=\max_{x\in C(f)}f(x)$,
	 are connected and diffeomorphic to the unit circle $S^1$.
	\end{corollary}
	\begin{proof}
	 The statement follows by combining Theorem \ref{th:05.01.2022.1} with the classification theorem of $1$-dimensional manifolds 
	 and the Non-Critical Neck Principle \cite[p. 194]{Palais-Terng}.
	\end{proof}
	\begin{corollary}\label{cor:05.01.2022.1}
	 The level curve $f^{-1}(h_{\max}(f))$ of a function $f$ subject to the hypothesis of Theorem \ref{th:05.01.2022.1} is convex, 
	 whenever $h_{\max}(f)\geq \mu_{\max}(f)$.
	\end{corollary}
\begin{proof}
The level curve $f^{-1}(h_{\max}(f))$ remains convex, as it is a connected regular curve bounding the convex sublevel set 
\begin{equation}\label{eq:05.01.2022.2}
f^{-1}\left(-\infty,h_{\max}(f)\right].  
\end{equation}
The convexity of \eqref{eq:05.01.2022.2} follows immediately out of its representation
\begin{align}\label{eq:27.12.2021.3}
f^{-1}\left(-\infty,h_{\max}(f)\right] & =\bigcap_{n\geq 1}f^{-1}\left(-\infty,h_{\max}(f)+\frac{1}{n}\right],\nonumber  
\end{align}
and the convexity of all sublevel sets
\[
f^{-1}\left(-\infty,h_{\max}(f)+\frac{1}{n}\right], \ n\geq 1.
\]
In fact $f^{-1}(h_{\max}(f))$ is the first convex level curve of $f$ completely contained in ${\rm cl}~{\rm Hess}^+(f)$.
\end{proof}

	\section{The {\rm Hess}$^+$ region and the levels of a product of functions}\label{SomeProductsOfFunctions}

	\begin{example}\label{ex:15.01.2022.1}
		We will describe the region ${\rm Hess}^+(f_ag_a)$, where $f_a,g_a:\mathbb{R}^2\longrightarrow\mathbb{R}$ are given by  \[f_a(x,y)=(x^2+y^2)^2+2a^2(x^2-y^2)
		\mbox{ and }g_a(x,y)=(x^2+y^2)^2+2a^2(x^2-y^2).
		\] 
		This can be done through the characterization 
		$$p\in {\rm Hess}^+(f_ag_a)\Longleftrightarrow{\rm Tr}(H_p(f_ag_a))>0\mbox{ and }\det(H_p(f_ag_a))>0.$$ 
		Elementary, but not short, computations show that 
		\begin{align}
		{\rm Tr}H_{(x,y)}(fg) & =f(x,y)\Delta g+g(x,y)\Delta f+2f_xg_x+2f_yg_y=64s^3-64dst-32sa^2b^2\nonumber\\
		\det H_{(x,y)}(fg) & = 64(7s^6+14s^4dt+11s^2d^2t^2  -4s^4d^2-12a^4b^4t^2-28a^2b^2s^4\nonumber\\
		& -4a^2b^2s^2dt + 36a^2b^2s^2t^2+12a^2b^2t^3d ),\nonumber
		\end{align}
 where $s=x^2+y^2,t=x^2-y^2,d=b^2-a^2$. Therefore		
		${\rm Tr}(H_{(x,y)}(f_ag_a))>0$ if and only if $x^2+y^2>\frac{a^2}{\sqrt{2}}$ and $\det(H_p(f_ag_a))>0$ if and only if $h(x,y)>0$, where
		\[h(x,y)=7(x^2+y^2)^6-12a^8(x^2-y^2)^2-28a^4(x^2+y^2)^4+36a^4(x^2+y^2)^2(x^2-y^2)^2.\]
		Moreover the restriction of the determinant 
		\[D(f_ag_a):=
		\begin{vmatrix} (f_ag_a)_{xx}  & (f_ag_a)_{xy}& (f_ag_a)_x \\ (f_ag_a)_{xy} & (f_ag_a)_{yy} & (f_ag_a)_y \\ (f_ag_a)_x & (f_ag_a)_y & 0\end{vmatrix}
		\]
		to the level set $f_ag_a=b$ is
		\[-2^8[3a^4(x^2+y^2)^8+5b(x^2+y^2)^6-12a^4b(x^2+y^2)^4-3b^2(x^2+y^2)^2+a^4b^2].
		\]
		In order to find the extreme points of the restriction of $D(f_ag_a)$ to the level set $f_ag_a=b$ we consider the associated Lagrange function \[\mathcal{L}(x,y,\lambda)=3a^4s^8+5bs^6-12a^4bs^4-3b^2s^2-\lambda(s^4-4a^4t^2-b)\]
		where $s$ stands for $x^2+y^2$ and $t$ stands for $x^2-y^2$.
		We are only interested about these local extremma for $b>0$, as $b=0$ is a critical value of the product $f_ag_a$ and for $b<0$ the level curve $f_ag_a=b$ is not connected. Indeed, the intersection of this level curve with the line $y=x$ is empty for $b<0$, as can be easily checked, and its intersection with the line $y=mx$ is not empty for $m\neq 1$ suitably chosen and the points of this intersection are located on the two sides of the line $y=x$. If on the contrary $b>0$, then the $b$-level  curve of the product $f_ag_a$ is regular and diffeomorphic to the unit circle via the diffeomorphism
		\[
		f:S^1\longrightarrow (f_ag_a)^{-1}(b), \ f(u,v)=\sqrt[4]{2a^4(v^2-u^2)+\sqrt{4a^8(v^2-u^2)^2+b}}\cdot(u,v)
		\]
		\begin{figure}[ht]
			\includegraphics[width=12cm]{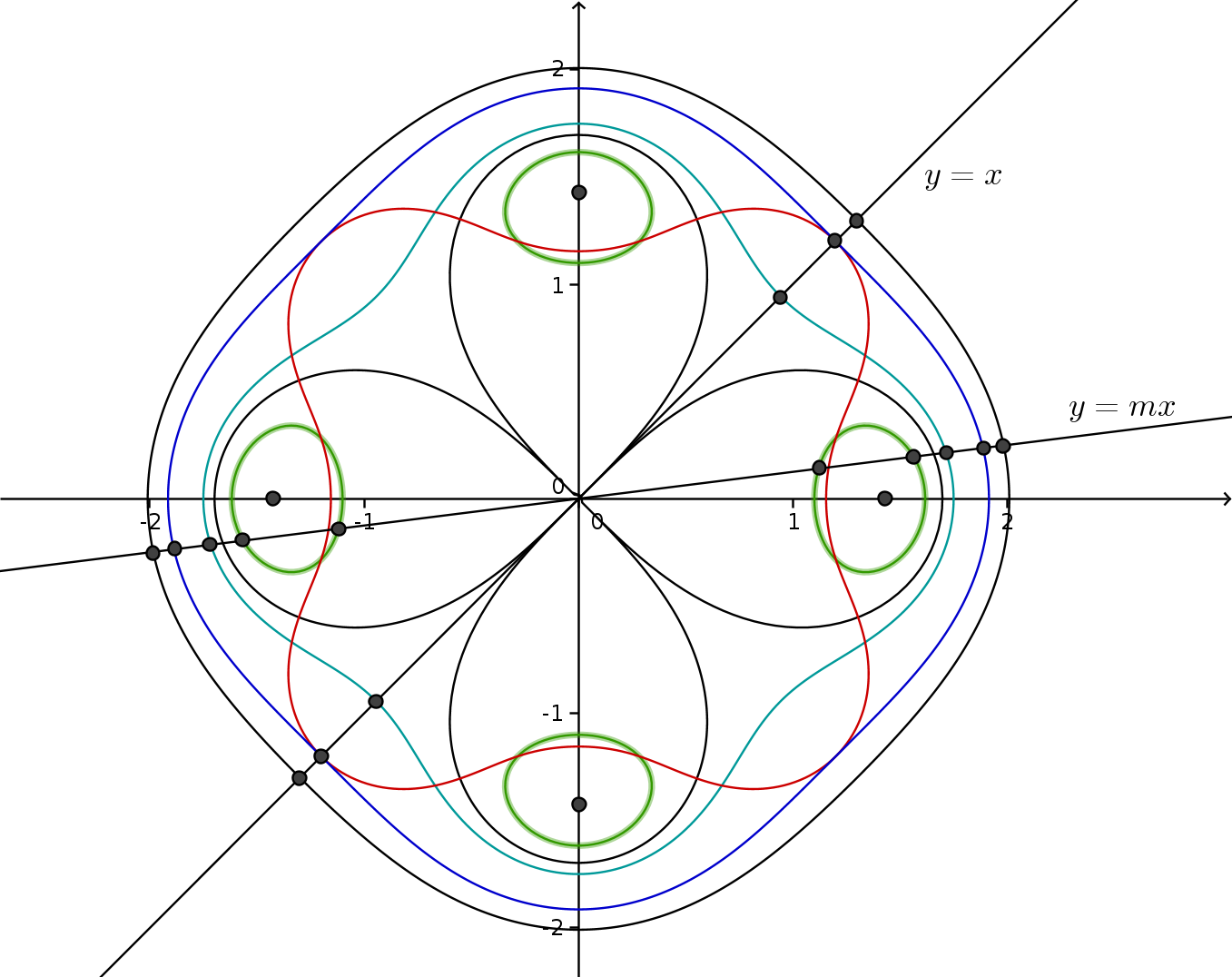}
			\caption{\textcolor{green}{A non-connected negative regular level of the product $f_ag_a$ with four components}, the critical zero level curve, \textcolor{cyan}{a connected non-convex level curve diffeomorphic to the unit circle $S^1$}, \textcolor{red}{the boundary of ${\rm Hess}^+(f_ag_a)$},  \textcolor{blue}{the first positive connected regular and convex level of $f_ag_a$} and a higher connected regular and convex level of $f_ag_a$ ($a=1.2$)}
			\label{ConnectedAndNonconnectedLevelSetsOfAProductOfFunctions}
		\end{figure}
		The equation 
		$\nabla \mathcal{L}(x,y,\lambda)=0$ is equivalent to the system
		\begin{equation}\label{eq01.09.2021.1}
			\begin{cases} x(48a^4s^7+60bs^5-96a^4bs^3-12b^2s-8\lambda s^3+16\lambda a^4t)=0 \\ y(48a^4s^7+60bs^5-96a^4bs^3-12b^2s-8\lambda s^3-16\lambda a^4t)=0 \\ s^4-4a^4t^2-b=0,\qquad \qquad 
			\end{cases} 
		\end{equation}
		which is further equivalent with
		\[
		\left\{
		\begin{array}{ll}
		x=0 \\ y=0\\
		b=0,\qquad \qquad 
		\end{array}
		\right.\!\!\!\!\!\!\!\!\!\!\!\!\!\!\!\!\!\!\!\!\!\!\!\!\mbox{ or }
		\left\{
		\begin{array}{ll}
		x=0 \\48a^4s^7+60bs^5-96a^4bs^3-12b^2s-8\lambda s^3+16\lambda a^4s=0 \\ y^8-4a^4y^4-b=0
		\end{array}
		\right.
		\]
		or 
		\[
		\left\{\begin{array}{ll}
		48a^4s^7+60bs^5-96a^4bs^3-12b^2s-8\lambda s^3+16\lambda a^4s=0=0 \nonumber\\
		y=0 \nonumber\\ 
		x^8-4a^4x^4-b=0\nonumber
		\end{array}
		\right.
		\]
		or 
		\[
		\left\{
		\begin{array}{ll} 48a^4s^7+60bs^5-96a^4bs^3-12b^2s-8\lambda s^3+16\lambda a^4t=0\nonumber \\ 48a^4s^7+60bs^5-96a^4bs^3-12b^2s-8\lambda s^3-16\lambda a^4t=0 \nonumber\\ s^4-4a^4t^2-b=0.\nonumber
		\end{array}
		\right.
		\]
		The solution set of the system \eqref{eq01.09.2021.1} 

		for $b>0$ is
		\[\left\{\left(\pm \frac{\sqrt[8]{b}}{\sqrt{2}},\pm \frac{\sqrt[8]{b}}{\sqrt{2}},6b(\sqrt[4]{b}-a^4)\right),(\alpha,0,\lambda),\left(0,\alpha,\lambda\right)\right\}
		\]
		where 
		\[
		\alpha=\pm\sqrt[4]{2a^4+\sqrt{4a^8+b}}\mbox{ and }\lambda= 6 
		\frac{(2a^4\alpha^4+b)^2}{\sqrt{4a^8+b}}.
		\]
	\end{example}
	Recall that the restriction of the determinant function $D(f_ag_a)$ to the level curve $f_ag_a=b$ is
	\[-2^8[3a^4(x^2+y^2)^8+5b(x^2+y^2)^6-12a^4b(x^2+y^2)^4-3b^2(x^2+y^2)^2+a^4b^2]
	\]
	and its values over the solution set of the system \eqref{eq01.09.2021.1}, for $b>0$, are
	\[
	2^9b^2(4a^4-\sqrt{b}),\mbox{ and }-2^9 (4 a^8 + b) (24 \alpha^4 a^8 + 6 a^4 b + \alpha^4 b).
	\]
	Since the second value is obviously negative, the determinant $D(f_ag_a)$ is not changing its sign on the $b$-level curve of the product 
	$f_ag_a$ if and only if \[4a^4-\sqrt{b}<0\Leftrightarrow b>16a^8.\] 
	In other words the $b$-level curve of the product $f_ag_a$ is convex whenever $b\geq 16a^8$ and $f_ag_a^{-1}(16a^8)$ 
	is the first convex level curve of $f_ag_a$, as can be easily show by using a similar argument with the one used for the proof of Corollary 
	\ref{cor:05.01.2022.1}.
	\begin{figure}[ht]
			\includegraphics[width=12cm]{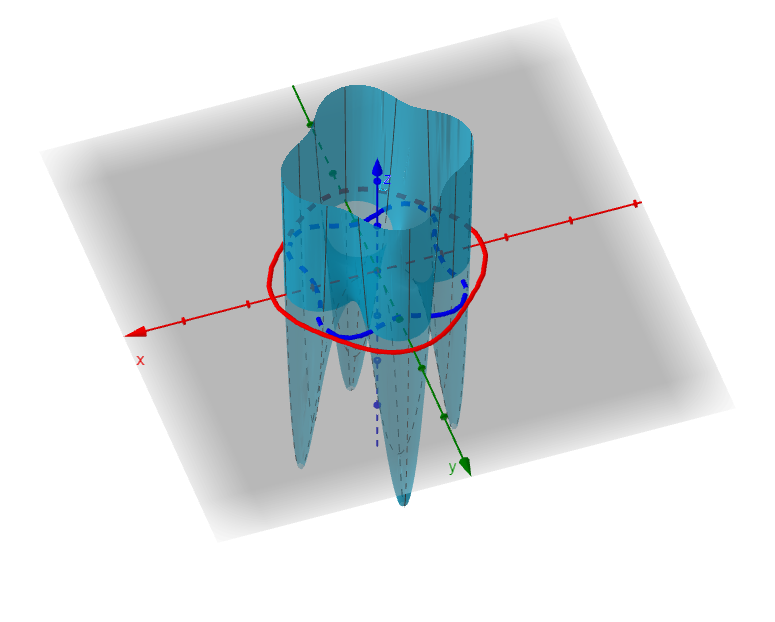}
			\caption{A piece of the graph of $f_1g_1$. \textcolor{blue}{The curve $\partial{\rm Hess}^+(f_1g_1)=h^{-1}\{0\}$}
				and \textcolor{red}{the first convex level set $(f_1g_1)^{-1}(16)$}}
			\label{Graph-of-a-product}
		\end{figure}
	
	\begin{remark} 
	For the function $f_ag_a$, the inequality $h_{max}(f_ag_a)\ge \mu_{max}(f_ag_a)$ holds. Indeed, on one hand we have $\mu_{max}(f_ag_a)=\max_{x\in C(f_ag_a)} (f_ag_a)(x)=\max B(f_ag_a)=0,$ and on the other hand, $h_{max}(f_ag_a)=\max\{(f_ag_a)(x)\mid x\in \mathbb{R}^2\setminus {\rm Hess^+}(f_ag_a)\}\ge (f_ag_a)(0,0)=0,$ as 
	\[
	{\rm Tr}~H_{(0,0)}(f_ag_a)=\det H_{(0,0)}(f_ag_a)=0
	\]
	and $(0,0) \in  \mathbb{R}^2\setminus {\rm Hess^+}(f_ag_a)$ therefore.
	The inequalities \[h_{max}(d_p^2d_q^2)\ge \mu_{max}(d_p^2d_q^2)\mbox{ and }h_{max}(f^2g^2)\ge \mu_{max}(f^2g^2),
	\]
	where $d_p:\mathbb{R}^2\longrightarrow\mathbb{R}$, $d_p(x)=\|x-p\|$ and $f=\sqrt{1+d_p^2}$, $g=\sqrt{1+d_q^2}$, $p,q\in\mathbb{R}^2$ also hold
	and will be treated in a forthcoming paper.
	\end{remark}

	\end{document}